\documentclass[a4paper,12pt]{amsart}
\usepackage{amsfonts}
\usepackage{amsmath,amssymb,latexsym,amsfonts,amscd}

\title[On quasismooth weighted complete intersections]{On
quasismooth weighted \\ complete intersections}
\author{Jheng-Jie Chen \and Jungkai A. Chen \and Meng Chen}

\address{\rm Department of Mathematics, National Taiwan University, Taipei, 106,
Taiwan} \email{d94221006@ntu.edu.tw}

\address{\rm Department of Mathematics, National Taiwan University, Taipei, 106,
Taiwan} \email{jkchen@math.ntu.edu.tw}

\address{\rm Institute of Mathematics, Fudan University,
Shanghai, 200433, People's Republic of China}
\email{mchen@fudan.edu.cn}

\thanks{The second author was partially supported by TIMS, NCTS/TPE
and National Science Council of Taiwan. The third author was
supported by both the National Outstanding Young Scientist
Foundation (\#10625103) and the NNSFC (\#10731030)}

\newcommand{\bQ}{{\mathbb Q}}
\newcommand{\OO}{\mathcal O}
\newcommand{\PPP}{\mathbb P}

\newcommand\CC{{\mathbb{C}}}
\newcommand\ZZ{{\mathbb{Z}}}

\DeclareMathOperator{\Codim}{codim}
\DeclareMathOperator{\Proj}{Proj}
\DeclareMathOperator{\rk}{rank}

\newtheorem{thm}{Theorem}[section]
\newtheorem{lem}[thm]{Lemma}
\newtheorem{cor}[thm]{Corollary}
\newtheorem{prop}[thm]{Proposition}
\newtheorem{claim}[thm]{Claim}

\newtheorem{conj}[thm]{Conjecture}
\theoremstyle{definition}

\newtheorem{setup}[thm]{}

\newtheorem{rem}[thm]{Remark}
\theoremstyle{remark}

\begin{document}
\begin{abstract} We prove two conjectures on weighted complete
inter\-sections (cf.\ \cite{Fletcher}) and give the complete
classification of threefold weighted complete intersections in weighted
projective space that are canonically or anticanonically embedded.
\end{abstract}

\maketitle

\section{\bf Introduction}

Weighted Projective Space, or WPS for short, is a natural generalization
of projective space. Complete intersections in WPS are the source of
interesting examples in the study of algebraic varieties, especially in
higher dimensional birational geometry. Mori and Dolgachev studied the
structure of WPS systematically. Then Reid \cite{C3-f,YPG} and
Iano-Fletcher \cite{Fletcher} discovered many new and interesting
examples of surfaces and threefolds, giving several famous lists.

Following the route of study initiated by Reid and Iano-Fletcher, we are
interested in the following two conjectures in \cite{Fletcher}.


\begin{conj} (\cite[Conjecture~18.19, p.~171]{Fletcher})
\label{conjcod}
\begin{itemize} \item[(1)] There are no canonically embedded threefold
complete intersections in WPS of codimension $>5$.

\item[(2)] There are no anticanonically embedded $\bQ$-Fano threefold
complete intersections in WPS of codimension $>3$.
\end{itemize}
\end{conj}

\begin{conj} (\cite[Conjecture~15.2,
p.~151]{Fletcher})
\label{completeness} The lists of threefold weighted
complete intersections \cite[15.1, 15.4, 16.6, 16.7, 18.16]{Fletcher}
are complete without any degree constraints.
\end{conj}


The aim of this note is to prove the above conjectures. Given a
weighted complete intersection $X=X_{d_1,\dots,d_c}\subset
\PPP(a_0,\dots,a_n)$, its {\em amplitude} is defined to be
$\alpha:=\sum_j d_j-\sum_i a_i$, so that $\omega_X \cong
\OO_X(\alpha)$. In fact, the answer to the first conjecture, under
quasismooth assumption, can be put in a more general form.

\begin{thm} \label{cod}
Let $X=X_{d_1,\dots,d_c}\subset \PPP(a_0,\dots,a_n)$ be a
quasi\-smooth weighted complete intersection of amplitude $\alpha$
and codimension $c$, not an intersection with a linear cone.

Then
$$
c\le\begin{cases} \dim X+\alpha+1 & \text{if } \alpha \ge 0, \\
\dim X & \text{if } \alpha < 0.
\end{cases}
$$
\end{thm}

Theorem~\ref{cod} implies Conjecture~\ref{conjcod} at least when $X$
is assumed to have at worst terminal quotient singularities, for
such threefolds are quasismooth (see \ref{qsmooth}). Another feature
of Theorem~\ref{cod} is that it holds for all dimensions, not only
for dimension 3.
In the course of proving Conjecture~\ref{completeness} (see
Theorem~\ref{Fano} and Theorem~\ref{gt123}), we also list the canonically
embedded threefold weighted complete intersections of codimension~4
and~5 (see Theorem~\ref{caya1234} and Corollary~\ref{cod45}).

We now outline our idea.
Let $X=X_{d_1,\dots,d_c}\subset
\PPP=\PPP(a_0,\dots,a_n)$ be a weighted complete intersection.
$X\subset \PPP$ is said to be {\em quasi\-smooth} if its affine cone
$C_X:=\pi^{-1}(X)$ is smooth away from $0$, where $\pi$ is the
natural quotient map from $\CC^{n+1}-\{0\}$ to $\PPP$. Let
$\delta_j:=d_j-a_{j+\dim X}$.

The paper, following Iano-Fletcher, divides into two parts. Part I
studies $N$-folds $X$ with $n=c+N$; then we aim for quasismooth, and
so are allowed to use restrictions given by the Jacobian matrix
restricted to coordinate strata. We have managed to find some
interesting inequalities between $\delta_j$ and $a_i$. Then
Theorem~\ref{cod} follows.

Part II is devoted to the classification of threefold weighted
complete intersections with at worst terminal singularities. For any
positive integer $i$, we define $\mu_i:=\#\{a_j\bigm|a_j=i\}$, and
$\nu_i:=\#\{d_j\bigm| d_j=i\}$. By Theorem~\ref{cod}, we have $\sum
\mu_i \le \alpha+A$ and $\sum \nu_i \le \alpha+B$, where $A,B$ are
small positive integers. Our algorithm is composed of the following
main steps.
\begin{itemize}

\item[1.] We exhaust tuples
$(\mu_1,\dots,\mu_6;\nu_2,\dots,\nu_6)$ when $\alpha=1$ (resp.\
$(\mu_1,\dots,\mu_5;\nu_2,\dots,\nu_5)$ when $\alpha=-1$) that satisfy
$$
\sum_{i=1}^h \mu_i \le \alpha+A \quad \text{and} \quad \sum_{i=2}^h
\nu_i \le \alpha+B
$$
where $h=6$ (resp. $h=5$).
\item[2.] We introduce a relation $\succ$ on formal baskets of orbifold
points (see 5.5), and reduction sequences on them (see 5.3), and
classify the initial cases of those sequences with given $\mu_i$,
$\nu_i$. Thus we get a complete list of possibilities for these formal
baskets.

\item[3.] For a given formal basket, one can compute the plurigenera by
Reid's Riemann--Roch formula. With Reid's ``table method'', one can
determine if the given formal basket really occurs on a weighted
complete intersection. Therefore, we are able to do a complete
classification,  which completes the original work initiated by Reid and
Iano-Fletcher.
\end{itemize}

A priori, there might be infinitely many initial baskets with given
$\mu_i,\nu_i$ since the index of each single basket might be arbitrarily
large. In the Fano case with $\alpha=-1$, we use the pseudo-effectivity
of $c_2(X)$ to obtain the maximal index of the basket. In the general
type case, we mainly use the fact $K_X^3>0$, while some exceptional
cases are more subtle. Anyway, we are able to show the finiteness of the
set of initial baskets, hence the finiteness of the set of formal
baskets.

The paper is organized as follows. Section~2 recalls some
definitions and notions of WPS and weighted  complete intersections
(w.c.i. for short). We prove Theorem~\ref{cod} in Section~3. It is
then possible to classify Calabi--Yau threefold complete
intersections, which we do in Section~4. In Section~5, we recall
some notions and properties of formal baskets that we need.
Section~6 is devoted to a detailed explanation of our algorithm for
the Fano case. We study weighted canonical threefolds in Section~7.
The codes of our program, written in MAPLE, is available upon
request.

\section{\bf Background material}

In this section, we recall some notions and properties that we need. Let
$a_0,\dots,a_n$ be positive integers. We view
$S=\mathbf{K}[x_0,x_1,\dots,x_n]$ as a graded ring over a field
$\mathbf{K}$ graded by $\deg x_i=a_i \in \mathbb{N}$ for each $i$.
Weighted projective space $\PPP=\PPP(a_0,a_1,\dots,a_n)$ is
the projective scheme or variety $\Proj S$ (in the sense of
\cite{Hart}). There is a natural quotient map
$\pi\colon\mathbb{A}^{n+1}_\mathbf{K}-\{0\} \to \PPP$ that
identifies $(x_0,\dots,x_n)$ with
$(\lambda^{a_0}x_0,\dots,\lambda^{a_n}x_n)$ for all $\lambda \in
\mathbf{K}^*$.

Let $T=\mathbf{K}[y_0,y_1,\dots,y_n]$ be the polynomial ring with
the
usual grading. By considering the homomorphism
$\tau\colon S \to T$ of
graded rings given by $\tau(x_i)=y_i^{a_i}$, one obtains
a finite
morphism $\tau \colon \PPP^n \to \PPP$. It follows that
$\PPP$ has at worst orbifold points along the coordinate strata.

\begin{setup}\label{symbol}{\bf Notation and conventions}
\begin{itemize}
\item[(1)] Fix an index set $I \subset \{0,\dots,n\}$ and
define
$\Pi_I:=\bigcap_{j \not \in I}\{x_j=0 \}$. In particular if
$I=\{i\}$,
we simply write $P_i$ in place of $\Pi_I$ and if $I=\{i,j\}$,
we write $P_iP_j$ for $\Pi_{\{i,j\}}$.

\item[(2)]  A weighted projective space $\PPP(a_0,a_1,\dots,a_n)$
is {\em well formed} if g.c.d.
$(a_0,\dots,\widehat{a_i},\dots,a_n)=1$ for all $i$. We assume
throughout that $\PPP$ is well formed. Most often, our complete
intersections are assumed to be also well formed (see \cite[6.10,
6.11, 6.12]{Fletcher} for explicit definitions).

\item[(3)] Let $b_1,\dots,b_n$ be integers and $r$ a positive integer.
An orbifold point $Q\in X$ is of type $\frac{1}{r}(b_1,\dots,b_n)$ if it
is analytically iso\-morphic to a quotient of $(\mathbb{A}^n,0)$ by an
action of the cyclic group $\ZZ/r$ of the form:
$$
\varepsilon (x_1,\dots,x_n)=(\varepsilon^{b_1}x_1,\dots,\varepsilon^{b_n}x_n),
$$
where $\varepsilon$ is a fixed primitive $r$th root of $1$.

\item[(4)] Let $d_1,\dots,d_c$ be positive integers and $f_1,\dots,f_c$
general homogeneous polynomials of degree $d_1,\dots,d_c$. By a complete
intersection $X=X_{d_1,\dots,d_c}\subset
\PPP=\PPP(a_0,\dots,a_n)$ we mean a subvariety defined by
$(f_1=f_2=\cdots=f_c=0)$.

\item[(5)]  A general hypersurface $X=\{f=0\}$ in $\PPP(a_0,\dots,a_n)$ is called a
linear cone if $f$ has degree $a_i$ for some $i$.

\item[(6)]  If $X=X_{d_1,\dots,d_c} \subset \PPP=\PPP(a_0,\dots,a_n)$
with $\deg f_c= a_n$, we can assume that $f_c=x_n+g(x_0,\dots,x_{n-1})$ for general $f_c$. 
Then one sees that $X$ is iso\-morphic to some
$X'=X'_{d_1,\dots,d_{c-1}}\subset \PPP'=\PPP(a_0,\dots,a_{n-1})$. Thus one sees that the generator $x_n$ and the 
equation $f_c$ is redundant. Therefore, 
we may always assume that our weighted complete inter\-section in
question is not an intersection with a linear cone. That is, numerically,
$$
a_i \neq d_j\quad \text{for all}\quad i,j.
$$
\end{itemize}
\end{setup}


\begin{setup}\label{qsmooth}{\bf Quasismoothness}.

Let $X_{d_1,\dots,d_c} \subset \PPP(a_0,\dots,a_{n})$ be a
complete intersection in WPS. Recall that the weighted projective
space is defined by the projection
$$\pi \colon \mathbb{C}^{n+1}-\{0\} \to \PPP(a_0,\dots,a_n),$$
Let $C_X:=\pi^{-1}(X)$ be the affine cone, where
$X:=X_{d_1,\dots,d_c}$. We say that $X$ {\it is quasi\-smooth} if
$C_X$ is smooth away from $0$. Note that, whenever the polarizing
divisor on $X$ is $\pm K_X$, quasismooth is equivalent to saying
that $X$ has only cyclic terminal orbifold points
$\frac{1}{r}(1,-1,b)$.
\end{setup}


\begin{setup}{\bf Weights and degrees}.

A weighted complete intersection $X_{d_1,\dots,d_c}\subset
\PPP(a_0,\dots,a_n)$ is said to be {\it normalized} if
$$
d_1\leq d_2\leq \cdots\leq d_c\quad \text{and}\quad
a_0\leq a_1 \le \cdots\leq a_n. \eqno{(2.1)}
$$
One can always normalize a weighted complete
intersection by renumbering the indices.
Given a normalized weighted complete intersection, for each
$j=1,\dots,c$, set $$\delta_j:= d_j -a_{j+\dim X},$$ and
$$\alpha:=\sum_{j=1}^c d_j-\sum_{i=0}^n a_i,$$
where $\alpha$ is called the {\it amplitude} of $X$. Then one has
$$
\delta:= \sum_{i=1}^c \delta_i=a_0+a_1+\cdots+a_{\dim X}+\alpha.
\eqno(2.2)
$$

If $X$ is quasi\-smooth and normalized, then \cite[Lemma
18.14]{Fletcher} says that
$$
d_j > a_{j+\dim X}\quad \mbox{for all}\quad j=1,\dots,c. \eqno{(2.3)}
$$
It follows in particular that
$$
\delta_j \geq 1 \quad \text{for all $j=1,\dots,c$}
\quad\hbox{and so}\quad \delta \ge c.
\eqno{(2.4)}
$$
\end{setup}

Moreover, we can also get some estimates on the weights and degrees.
Suppose $X=X_{d_1,\dots,d_c}$ in $\PPP(a_0,\dots,a_{c+\dim X})$ is
a quasi\-smooth terminal weighted complete intersection with
amplitude $\alpha\geq 0$.  Write $d_j=\lambda_j a_{j+\dim X}$
for all $j=1,\dots,c$. Then the inequalities (2.3) imply
$\lambda_j>1$ for all $j$. We claim that:
$$
\sum_{j=1}^c \lambda_j \le c+\alpha+\dim X +1. \eqno{(2.5)}
$$

In fact
\begin{align*}
(\dim X+1) a_{\dim X+1}+\alpha &\ge a_0+\cdots+a_{\dim X}+\alpha \\
&=\sum_{j=1}^c \delta_j = \sum_{j=1}^c (\lambda_j -1) a_{j+\dim X} \\
&\ge \sum_{j=1}^c (\lambda_j-1) a_{\dim X+1}
\end{align*}
gives
\begin{multline*}
\Bigl(\dim X+1+c +\alpha-\sum_{j=1}^c \lambda_j  \Bigr)a_{\dim X+1}\\
\ge \Bigl(\dim X+1+c-\sum_{j=1}^c \lambda_j \Bigr)a_{\dim X+1} +\alpha \ge 0.
\end{multline*}
This proves the inequality (2.5).

It follows that
$$
\frac{\left(\frac{c+\alpha+\dim X+1}{c}\right)^c}
{\prod_{i=0}^{\dim X}a_i } \ge
\frac{\left(\frac{\sum_{j=1}^c \lambda_j}{c}\right)^c}
{\prod_{i=0}^{\dim X}a_i }\geq
\frac{\prod_{j=1}^c \lambda_j }{\prod_{i=0}^{\dim X}a_i }
= \frac{\prod_{j=1}^c d_j}{\prod_{i=0}^{n} a_i}. \eqno(2.6)
$$

\section{\bf Part I: Quasismooth w.c.i. $N$-folds}

In this section, we go somewhat further than \cite{Fletcher} in the study
of quasi\-smoothness in order to prove our first theorem. We first present
the results of \cite[18.14]{Fletcher} and \cite[8.1, 8.7]{Fletcher} in the
following generalized form.

\begin{prop} \label{qs} Assume that the normalized complete
intersection
$$
X_{d_1,\dots,d_c} \subset \PPP(a_0,\dots,a_{n})
$$
is quasi\-smooth and not an intersection with a linear cone (i.e., $d_j\neq
a_i$ for all $i,j$).
\begin{enumerate}
\item If $a_{t} > d_1$ for some $t \ge 0$, then $a_t\mid d_j$ for some $j$.
In particular, $\delta_c  \geq a_n$ in this situation.

\item If $c\geq \dim X+1$, then
$$
\delta_{c-j}=d_{c-j}-a_{n-j}\geq a_{\dim X-j} \quad\text{for}\quad j=0,\dots,\dim X.
$$
\end{enumerate}
\end{prop}

\begin{proof} We first prove (1).

Let $x_0,\dots,x_n$ be the coordinates of degree $a_0,\dots,a_n$
respectively. Suppose that $a_t \nmid d_j$ for all $j$. Let $f_j$ be the
equation of degree $d_j$. By our assumption, we know that $f_1$ does not
involve $x_t$ and $f_j$ contains no monomials of the form $x_t^\mu$ with
$\mu>0$.

We consider $\Pi=(x_0=x_1=\cdots=\widehat{x_t}=\cdots=x_{n}=0)\subset
\mathbb{A}^{n+1}$, which is clearly nonempty. Then $X$ is not
quasi\-smooth if there are singularities in $C_X \cap \Pi$. In fact, for
general points in $\Pi$, we may consider the Jacobian matrix
$$
J=\left( \begin{array}{ccc} \frac{\partial f_1}{\partial x_0} &
\dots & \frac{\partial f_1}{\partial x_n} \\
\vdots & & \vdots \\
\frac{\partial f_c}{\partial x_0} & \dots & \frac{\partial f_c}{\partial x_n} \\
\end{array} \right).
$$
Notice that the first row is identically zero at the general point
of $\Pi$, hence $C_X \cap \Pi$ has at least one singularity, which
contradicts quasi\-smoothness. Hence $a_t\mid d_j$ for some $j$.

Since $a_n\geq a_t>d_1$, if we take $t=n$, then the proof says
$a_n\mid d_j$ for some $j$. Thus $\delta_c=d_c-a_{n}\geq d_j-a_n\geq
a_n$. Part (1) is proved.




{}For part (2), we assume that
$f_j$ is a polynomial of degree $d_j$ for each $j>0$.

Given an integer $j\in [0,\dim X]$, we suppose by contradiction that
$\delta_{c-j}=d_{c-j}-a_{n-j}<a_{\dim X-j}$ for some $j$. We hope to
deduce a contradiction. We have
$$
d_{j'} \leq d_{c-j} < a_{n-j}+a_{\dim X-j}
\quad\hbox{for all}\quad j'\leq c-j. \eqno(3.1)
$$

We consider $\Pi'=(x_0=x_1=\cdots=x_{n-j-1}=0)\subset \mathbb{A}^{n+1}$.
For each integer $m>0$, we may write
$$
f_{m}=h_{m}(x_{n-j},\dots,x_{n})+\sum_{i=0}^{n-1-j}
g^i_{m}(x_{n-j},\dots,x_{n})x_i+l_{m}(x_0,\dots,x_{n}),
$$
where we assume $\deg_{x_0,\dots,x_{n-1-j} }(l_{m}) \geq 2$.

By inequality (3.1), $d_{j'}<2 a_{n-j}$ for all $j' \leq
c-j$ and so $h_{j'}=0$. Also notice that $C_X \cap \Pi'$ is
defined by $x_0=x_1=\cdots=x_{n-1-j}=0$ and $h_{c-j+1}= \cdots
=h_{c}=0$, which is clearly nonempty of dimension $\geq
(j+1)-j=1$.

Again $X$ will not be quasi\-smooth if there are singularities
in $C_X \cap \Pi'$. In fact, for general $P\in C_X \cap \Pi'$, we
consider the Jacobian matrix
$$
J(P)=\left( \begin{array}{ccc} \frac{\partial f_1}{\partial x_0}(P) & \dots & \frac{\partial f_1}{\partial x_n}(P) \\
\vdots & & \vdots \\
\frac{\partial f_c}{\partial x_0}(P) & \dots & \frac{\partial f_c}{\partial x_n}(P) \\
\end{array} \right).
$$

When $j' \le c-j$, one has $h_{j'}=0$ as above, and hence
$$
\frac{\partial f_{j'}}{\partial x_i}(P)=
\begin{cases}{g^i_{j'}(P)} &\text{if } i \le n-j-1, \\ 0 &\text{if
} i \ge n-j.\end{cases}
$$
Moreover by inequality $(3.1)$, we have $g^i_{j'}=0$ for  $i \geq \dim
X-j$ (otherwise, $d_{j'}\geq a_{n-j}+a_{\dim X-j}$, contradicting (3.1)).
Thus we have seen $\frac{\partial f_{j'}}{\partial x_i}(P)=0$ for
$j'\leq c-j$ and $i\geq \dim X-j$. Noting that $\text{max}\{\dim X-j,
j\}\leq \dim X$, one sees that
$$
\rk(J(P)) \leq \dim X \leq c-1<c=\Codim C_X,
$$
so that $C_X$ is singular at $P$, a contradiction.
\end{proof}

%
\begin{setup}{Proof of Theorem~\ref{cod}.}
\end{setup}
\begin{proof} If $\alpha\geq 0$ and $c>\dim X+1+\alpha$, then
Proposition~\ref{qs}, (2) and inequality (2.3) give the following relation:
\begin{align*}
a_0+\cdots+a_{\dim X}+\alpha &=\sum_{i=c-\dim X}^{c} \delta_i+\sum_{i=1}^{c-\dim X-1} \delta_i\\
& \geq a_0+\cdots+a_{\dim X}+\sum_{i=1}^{c-\dim X-1} \delta_i \\
& >a_0+\cdots+a_{\dim X}+\alpha,
\end{align*}
a contradiction.

In the $\bQ$-Fano case, we have $\alpha<0$. Suppose that
$c >\dim X+1+\alpha$. Again Proposition~\ref{qs}, (2) gives
$$
a_0+\cdots+a_{\dim X}+\alpha=\sum_{i=1}^c \delta_i\geq a_0+\cdots+a_{\dim X},
$$
a contradiction.
\end{proof}

A direct consequence is the following:
\begin{cor}\label{C1} Conjecture~\ref{conjcod} is true for canonically
(resp.\ anticanonically) polarized threefold complete intersections
with terminal quotient singularities.
\end{cor}
\begin{proof}  Since $X$ has
terminal quotient singularities, $X$ is quasismooth by
\ref{qsmooth}. Thus the statement follows.
\end{proof}


\begin{rem} It is a very interesting open problem to prove Conjecture~\ref{conjcod}
without the ``terminal quotient'' assumption.
\end{rem}

\section{\bf Part II--1: General properties of w.c.i. threefolds and
the classification of weighted terminal Calabi--Yau threefolds}

\bigskip

Aiming at proving Conjecture~\ref{completeness}, we begin to
concentrate our study on threefolds, i.e., the case $n=c+3$.

Assume moreover that $X_{d_1,\dots,d_c}\subset\PPP(a_0,\dots,a_{c+3})$
has only isolated singularities. One can easily determine some numerical
properties on $X_{d_1,\dots,d_c}$ as follows.

\begin{prop} \label{iso_sing}
Let $X=X_{d_1,\dots,d_c}$ be a quasi\-smooth weighted complete
intersection. Suppose that $X$ has only isolated singularities.
Then:
\begin{enumerate}
\item for all $\mu=1,\dots,c+1$, the greatest common divisor of
any $\mu$ of the $\{a_i\}$ must divide at least $\mu-1$ of the
$\{d_j\}$;

\item for all $\mu>c+1$, the greatest common divisor of any $\mu$
of the $\{a_i\}$ must be $1$.
\end{enumerate}
\end{prop}

\begin{proof} Given $h$ the greatest common divisor of  $\mu$ of the
$\{a_i\}$.  After renumbering, we may assume that
$h=(a_0,\dots,a_{\mu-1})$. Let $f_j$ be the general homogeneous
polynomial of degree $d_j$ for $j=1,\dots,c$. We may write
$$ f_j = h_j(x_0,\dots,x_{\mu-1})+
\sum_{i=\mu}^{n} g^i_j(x_0,\dots,x_{\mu-1}) x_i+l_j(x_0,\dots,x_n), \eqno(3.2)$$
where $\deg_{x_\mu,\dots,x_n}l_j \ge 2$.

To prove (1), we assume that $h>1$ since, otherwise, there is
nothing to prove. Suppose, on the contrary, that $h$ divides at
most $\mu-2$ of the $\{d_j\}$. After renumbering, we may assume
that $h \nmid d_{\mu-1}, \dots, h \nmid d_c$. Then one sees that
$h_j=0$ for  $\mu-1 \le j \le c$.

Let $I=\{0,\dots,\mu-1\}$ and $\Pi_I=\{ x_\mu=\cdots=x_n=0\}$, which
is clearly of dimension $\mu-1$ and $\PPP(a_0,\dots,a_n)$ has
orbifold points along $\Pi_I$. Now $h_j=0$ implies that $\Pi_I
\subset (f_j=0)$ for all $ \mu-1 \le j \le c$. It follows that
$\Pi_I \cap X $ has dimension $\ge 1$ since it is cut out by at most
$\mu-2$ equations from $\Pi_I$. Noticing that $\Pi_I \subset
\text{Sing}(\PPP)$ and by \cite{Dim}, one has
$$\Pi_I \cap X \subset \text{Sing}(\PPP) \cap X = \text{Sing}(X)$$
a contradiction.

To see (2), one notices that $\Pi_I \cap X$ always has dimension
$\ge 1$. Hence $h=1$ since $X$ has only isolated singularities.
\end{proof}

Since threefold terminal singularities are isolated, we get the
following:

\begin{prop} \label{3term}
Let $X=X_{d_1,\dots,d_c}$ be a quasi\-smooth threefold weighted complete
intersection. Suppose that $X$ has at worst terminal singularities. For
$1 \le \mu \le c+1$, let $h$ be the greatest common divisor
of distinct
$a_{i_1},\dots,a_{i_\mu}$. Then one of the following holds:
\begin{enumerate}
\item $h$  divides at least $\mu$ of the $\{d_j\}$;

\item $h$  divides $\mu-1$ of the $\{d_j\}$ and $h\mid (a_m+\alpha)$
for some $m \ne a_{i_j}$, $j=1,\dots,\mu$.
\end{enumerate}
\end{prop}

\begin{proof}
By Proposition~\ref{iso_sing}, we know that $h$ divides at least
$\mu-1$ of the $\{d_j\}$. Suppose now that $h$ divides exactly
$\mu-1$ of the $\{d_j\}$. After renumbering, we may assume that
$h=(a_0,\dots,a_{\mu-1})$ and $h\mid d_i$ for $i=1,\cdots, \mu-1$
and $h\nmid d_\mu,\dots,h \nmid d_c$.

We may write $f_j$ as in (3.2).  For $ \mu \le j \le c$, we have
 $h_j=0$ as in the proof of Proposition~\ref{iso_sing}.

Let $I=\{0,\dots,\mu-1\}$ and $\Pi_I=\{ x_\mu=\cdots=x_n=0\}$ as
above. One notices that $\Pi_I \cap X \ne \emptyset$. For any point
$P \in \Pi_I \cap X$, one may check the quasi\-smoothness near $P$.
$$
J(P)=\left( \begin{array}{ccc} \frac{\partial f_1}{\partial x_0}(P)
 & \dots & \frac{\partial f_1}{\partial x_n}(P) \\
\vdots & & \vdots \\
\frac{\partial f_c}{\partial x_0}(P) & \dots & \frac{\partial f_c}{\partial x_n}(P) \\
\end{array} \right).
$$

For $\mu \le j \le c$,
$$
\frac{\partial f_{j}}{\partial x_i}(P)=
\begin{cases}
g^i_{j}(P) &\text{if } i \ge \mu,\\
0 &\text{if } i \le \mu-1.
\end{cases}
$$

Since $X$ is quasi\-smooth at $P$, that is, $J(P)$ of full rank. It
follows that the submatrix
$$
\left( \begin{array}{ccc}  g_\mu^\mu(P) & \dots & g_\mu^n(P) \\
\vdots & & \vdots \\
g_c^\mu(P)) & \dots & g_c^n(P) \\
\end{array} \right)
$$
is of full rank. We may renumber the indices so that $g^j_j(P) \ne 0$
for $ \mu \le j \le c$. Hence for $\mu \le j \le c$, $f_j$ contains at
least a nonzero monomial of the form $x_0^{n_0}\cdots
x_{\mu-1}^{n_{\mu-1}} x_j$. In particular, one has
$$
d_j=n_0a_0+\cdots+n_{\mu-1}a_{\mu-1}+a_j \equiv a_j \ \
\ \ (\text{mod}\ h)
$$
for $\mu \le j \le c$.

In our case, $n-c=3$. By the Inverse Function Theorem, we may conclude
that the singularity at $P$ is of type
$\frac{1}{h}(a_{c+1},\dots,a_n)$.  In fact, by the Terminal
Lemma, after renumbering, we may assume that
$$a_{c+1}+a_{c+2} \equiv 0 \ \ \ (\text{mod}\ h).$$ Combining all
these, we have
$$ \alpha=\sum_{j=1}^c d_j - \sum_{i=0}^n a_i \equiv  \sum_{j=\mu}^c d_j -
\sum_{i=\mu}^n a_i  \equiv -\sum_{i=c+1}^n a_i \equiv -a_n\ \ \
(\text{mod}\ h).$$ This completes the proof.
\end{proof}

More refined properties can be realized when $\alpha=0$.

\begin{cor} \label{CY}
Let $X$ be a quasi\-smooth threefold weighted complete intersection with
$\alpha=0$. Suppose that $X$ has at worst terminal singularities. For $1
\le \mu \le c$, let $h$ be the greatest common divisor of
distinct
$a_{i_1},\dots,a_{i_\mu}$. Then $h$ divides at least $\mu$ of the
$\{d_j\}$.

Moreover, the greatest common divisor of distinct
$a_{i_1},\dots,a_{i_\mu}$
 with $\mu \ge c+1$ must be $1$.
\end{cor}

\begin{proof}
By Proposition~\ref{3term}, if $h$ does not divides $\mu$ of the
$\{d_j\}$, then $h\mid a_m$ for some $m \ne i_j$. Hence $h$ indeed
divides $\mu+1$ of the $\{a_i\}$. By Proposition~\ref{iso_sing},
one sees that $h$ divides at least $\mu$ of the $\{d_j\}$.

To see the second statement, it suffices to consider the case that
$\mu=c+1$ by Proposition~\ref{iso_sing}(2). Let $h$ be the
greatest common divisor of distinct $a_{i_1},\dots,a_{i_{c+1}}$.
By Proposition~\ref{3term}, $h$ must divides $a_m$
for some $m \ne a_{i_j}$, $j=1,\dots,c+1$. It follows that $h$ is
in fact a greatest common divisor of $a_{i_1},\dots,
a_{i_{c+1}},a_m$, which constitutes $c+2$ of $\{a_i\}$. By
Proposition~\ref{iso_sing}(2), one has $h=1$.
\end{proof}

\begin{cor} \label{CYint}
Let $X$ be a quasi\-smooth threefold weighted complete intersection
with
$\alpha=0$. Suppose that $X$ has at worst terminal
singularities. Then
$\prod a_i$ divides $\prod d_j$.
\end{cor}

\begin{proof}
For a prime factor $p$ of $\prod a_i$, set
$\beta_k:= \#\{ i \bigm| p^k \text{ divides } a_i\}$ and
$\gamma_k:=\# \{ j \bigm| p^k\text{ divides }d_j\}$. By
Corollary~\ref{CY}, one has $\beta_k \le
\gamma_k$ for all $k$. Therefore, $e_\beta:=\sum \beta_k \le \sum
\gamma_k =:e_\gamma$, where $e_\beta$ and $e_\gamma$ are the
exponents of $p$ in $\prod a_i$ and $\prod d_j$ respectively. This
completes the proof.
\end{proof}

In the rest of this section, we classify quasi\-smooth terminal
Calabi--Yau threefolds by Corollaries \ref{CY}, \ref{CYint} and
inequality (2.6).

\begin{thm}\label{caya1234} All quasi\-smooth, weighted terminal
Calabi--Yau threefolds (and not an intersection with a linear cone) are
as follows:
$$
\begin{array}{ll}
\mathrm{No.}\ 1 & X_{5}\subset \PPP(1,1,1,1,1),\\
\mathrm{No.}\ 2 & X_{6}\subset \PPP(1,1,1,1,2),\\
\mathrm{No.}\ 3 & X_{8}\subset \PPP(1,1,1,1,4),\\
\mathrm{No.}\ 4 &  X_{10}\subset \PPP(1,1,1,2,5);\\
\mathrm{No.}\ 5 &  X_{2,4}\subseteq \PPP(1,1,1,1,1,1),\\
\mathrm{No.}\ 6 & X_{3,3}\subseteq \PPP(1,1,1,1,1,1),\\
\mathrm{No.}\ 7 & X_{3,4}\subseteq \PPP(1,1,1,1,1,2),\\
\mathrm{No.}\ 8 & X_{2,6}\subseteq \PPP(1,1,1,1,1,3),\\
\mathrm{No.}\ 9 &  X_{4,4}\subseteq \PPP(1,1,1,1,2,2),\\
\mathrm{No.}\ 10 & X_{4,6}\subseteq \PPP(1,1,1,2,2,3),\\
\mathrm{No.}\ 11 &  X_{6,6}\subseteq \PPP(1,1,2,2,3,3),\\
\mathrm{No.}\ 12 & X_{2,2,3}\subseteq \PPP(1,1,1,1,1,1,1),\\
\mathrm{No.}\ 13 &  X_{2,2,2,2}\subseteq\PPP(1,1,1,1,1,1,1,1).\\
\end{array}
$$
\end{thm}
\begin{proof} Assume that the Calabi--Yau threefold
$$X:=X_{d_1,\dots,d_c}\subset \PPP(a_0,\dots,a_{c+3})$$ with amplitude
$\alpha=0$. Theorem~\ref{cod} says $c\leq 4$. By Corollary~\ref{CYint},
$\frac{\prod d_j}{\prod a_i}$ is an integer, hence in particular
$\frac{\prod d_j} {\prod a_i} \ge 1$. Together with $(2.6)$, one sees
that
$$
16 \ge \left(\frac{c+4}{c}\right)^c \ge a_0a_1a_2a_3. \eqno(4.1)
$$

In particular, $a_3 \le 16$. Notice that $a_n\mid d_j$ for some $j$ by
Corollary~\ref{CY}. Hence $$64 \ge a_0+a_1+a_2+a_3=\delta \ge
\delta_c=d_c-a_n \ge d_j-a_n \ge a_n.$$ It is thus possible to do a
complete classification.

Recall that Iano-Fletcher has shown his classification for the case
$c=1$ in \cite[14.3]{Fletcher}. Thus we continues with the cases with
$c\ge 2$.
\medskip

\noindent {\bf Case 1.} $c=2$.

By (4.1), we have $ {a_0a_1a_2a_3} \le 9.$ It follows that $a_0=1$.
In fact, we conclude that $a_1=1$ otherwise
$a_1=a_2=a_3=2$, contradicting Corollary~\ref{CY}.
\medskip

\noindent {\bf Subcase 1.1.} $a_3=1$.

By Corollary~\ref{CY}, $a_5\mid d_j$ for some $j=1,2$. It follows
that $d_j \ge 2 a_5 $ and hence
$$4 =\delta \ge 1+\delta_2 =1+d_2-a_5 \ge 1+ d_j-a_5 \ge 1+a_5.$$
Thus $3 \ge a_5$ and then our classification shows that $X$
corresponds to Nos.~5, 6, 7, 8, and~9 in the table.
\medskip

\noindent {\bf Subcase 1.2.} $a_3>1$ and $a_4=a_3$.

By Corollary~\ref{CY}, we have at least $a_5> a_4$ and $a_3> a_2$. It
follows that $a_5 \ge a_2+2$. By Corollary~\ref{CY}, we have
$a_3=a_4\mid d_1$, $a_3\mid d_2$ and $a_5\mid d_j$ for some $j$. Hence we have
$d_2 \ge 2 a_5$ and $d_1 \ge 2 a_4$. Therefore,
$$ 2+a_2+a_3=\delta = \delta_1+\delta_2 \ge a_4+a_5 \ge a_3+a_2+2.$$
So $a_3=a_4=a_2+1$, $a_5=a_2+2$ and $d_1=2(a_2+1), d_2=2(a_2+2)$.
Now $a_4 \mid d_2$ implies that $(a_2+1)\mid 2$. We must have $a_2=1$. This
gives No.~10.
\medskip

\noindent {\bf Subcase 1.3.} $a_3>1$ and $a_4>a_3$.

Suppose first that $a_5=a_4$; then $d_j \ge 2 a_5=2a_4$
for $j=1, 2$ by Corollary~\ref{CY}. Thus we have
$$
2+a_2+a_3 = \delta = \delta_1+\delta_2 \ge a_4+a_5.
$$
Hence $a_2=a_3$ , $a_4=a_3+1$ and $d_1=d_2=2a_5$. Notice that
$(a_2,a_4)=1$ and $a_2\mid d_1$. It follows that $a_2=2$.
This gives No.~11.

Suppose now that $a_5>a_4$. We show that this leads to a
contradiction. If $a_5\mid d_1$, then $\delta_2=d_2-a_5 \ge d_1-a_5 \ge
a_5$ and $\delta_1=d_1-a_4 \ge d_1-a_5 \ge a_5$. Hence
$2+a_2+a_3=\delta  \ge 2a_5 \ge 2a_3+4 \ge a_2+a_3+4$, a
contradiction. Thus we may assume that $a_5 \mid d_2$ and hence
$\frac{d_2}{a_5} \ge 2$. Moreover, we may assume that $d_2=2a_5$
otherwise $\delta_2 \ge 2a_5$ gives a contradiction again.

If $a_4\mid d_2=2a_5$, then $a_5=m\cdot \frac{a_4}{2}$ for some $m \ge
3$. Let $h:=(a_4,a_5)$, then $h= \frac{a_4}{2}$ or $h=a_4$.
Proposition~\ref{CY} implies $h\mid d_1$, and by (2.3), $d_1 \ge
\frac{3a_4}{2}$. We have
$$ 2+a_2+a_3= \delta_1+\delta_2 \ge \frac{a_4}{2}+a_5 \ge \frac{m+1}{2} {a_4} \ge \frac{m+1}{2} (a_3+1),$$
for some $m \ge 3$. Thus $m=3$, and $d_1=\frac{3a_4}{2}$. But now $a_5=d_1$, a contradiction.

If $a_4\nmid d_2$, then $a_4\mid  d_1$. We have $\delta_1 \ge a_4$. We thus have
$$ 2+a_2+a_3 = \delta_1+\delta_2 \ge a_4 +a_5 \ge 2a_3+3,$$
a contradiction.
\medskip

\noindent {\bf Case 2.} $c=3$.

This case is similar. We leave the details to readers.

\medskip

\noindent {\bf Case 3.} $c=4$.

By Proposition~\ref{qs}(2), we have
$$ a_0+\cdots+a_3=\delta_1+\cdots+\delta_4 \ge a_0+\cdots+a_3.$$
Therefore, $\delta_i=a_{i-1}$ for all $i=1,\dots,4$. In
particular, $d_4=a_7+a_3$.

Moreover, Proposition~\ref{CY} says $a_7\mid d_j$ for some $j$. It
follows that $a_3+a_7=d_4 \ge 2a_7$. Thus $a_3=\cdots=a_7$.
We must have that $a_7=1$ by Corollary~\ref{CY}. This gives No.~13
\end{proof}

\begin{rem} An interesting point is that all Calabi-Yau threefolds obtained in
Theorem~\ref{caya1234} are actually nonsingular.

\end{rem}

\begin{rem} This theorem is possibly known to experts. However, our method
is simple and extremely effective. \end{rem}

\section{\bf Part II--2: Baskets of orbifold points and the finiteness}

Our classification uses a very effective tool, ``basket analysis''. Here
we only recall some basic definitions and properties of baskets. In
particular, we introduce the notion of packing. All details can be found
in \cite[Section~4]{explicit}.

A {\it basket} $\mathcal{B}$ of singularities is a collection
(permitting weights) of terminal orbifold points of type
$\frac{1}{r_i}(1,-1,b_i)$ for $i\in I$, where $I$ is a finite indexing set,
$0<b_i\leq \frac{r_i}{2}$ and $b_i$ is coprime to $r_i$ for each
$i$. A {\it single basket} means a single singularity $Q$ of type
$\frac{1}{r}(1,-1,b)$. For simplicity, we will always denote a
single basket by $(b,r)$ or $\{(b,r)\}$. So we will simply write a
basket as:
$$\mathcal{B}:=\{n_i\times(b_i,r_i)\bigm|i\in I,n_i\in \mathbb{Z}^+\}.$$

\begin{setup}{\bf Packing.} We recall and slightly generalize the notion
of ``packing'' introduced in \cite[Section~4]{explicit}. Given a basket
$$
B=\{(b_1,r_1),(b_2,r_2),\dots,(b_k,r_k)\},
$$
we call the basket
$$B':=\{(b_1+b_2,r_1+r_2),(b_3,r_3),\dots,(b_k,r_k)\}$$
(or $B':=\{(b_1,r_1+1),(b_2,r_2),\dots,(b_k,r_k)\}$) a packing of
$B$, written as $B\succ B'$. The relation ``$\succ$''
automatically defines a partial ordering on the set of baskets. If
furthermore $b_1r_2-b_2r_1=1$ (or respectively, $b_1=1$), we call
$B\succ B'$ a $prime\ packing$.
\end{setup}

\begin{rem} The notion of packing defined in \cite{explicit} exclude
the second bracketed situation above, the case $r_2=1$. The reader can
check without difficulty the properties of packings of
\cite[Section~4]{explicit}, with this situation considered as the natural
generalization
$$
\{(0,1),(b_1,r_1)\} \succ \{ (b_1, r_1+1)\}.
$$
\end{rem}

\begin{rem} The notion of unpacking corresponds to some well-known
geometric constructions. For example, the Kawamata blowup $\pi: Y
\to X$ of a terminal cyclic quotient point of type
$\frac{1}{r}(a,r-a,1)$ gives two cyclic quotient singularities of
types  $\frac{1}{a}(-r,r,1)$ and $\frac{1}{r-a}(r,-r,1)$ on $Y$
respectively. It is easy to see that this gives an unpacking of the
basket of $X$. Moreover, Danilov's economic resolution gives a
series of unpacking of the basket (see for example \cite[3.4]{CPR},
\cite{KA2}).

However, for our combinatoric purpose here and in our previous
works, packing is more convenient.

\end{rem}

\begin{setup}
{\bf Canonical sequence of baskets.} Given a basket $B$, as
discovered in \cite[Section~4]{explicit} (see also \cite[2.5]{AIM}
for a brief definition), there is a unique (hence, called
``canonical'') sequence $\{B^{(m)}(B)\}$ of finite length with:


$$B^{(0)}(B)\succ B^{(5)}(B)\succ \cdots\succ B^{(n)}(B)\succ\cdots\succ B,$$
where the $B^{(0)}(B)=\{n_{1,r} \times (1,r)\}_{r \ge 2}$ and the
step $B^{(n-1)}(B) \succ B^{(n)}(B)$ can be achieved by totally
$\epsilon_n(B)$ prime packings of the type $\{(b_1,r_1),(b_2,r_2)\}
\succ \{(b_1+b_2,r_1+r_2)\}$ with $r_1+r_2=n$. The nonnegative
number $\epsilon_n(B)$ is computable in terms of the datum of $B$.


Clearly there are finitely many baskets dominated by a fixed
initial basket $B^{(0)}$.
\end{setup}





\begin{thm} (Reid \cite{C3-f, YPG})
For any projective $3$-fold $X$ with at worst canonical
singularities, there exists a basket $B(X)$ of singularities such
that, for all $m\in \mathbb{Z}$,
$$
\chi_m:=\chi(\OO_X(mK_X))
=\frac{(2m-1)m(m-1)}{12}K_X^3-(2m-1)\chi(\OO_X)+l(m),
$$
where the correction term $l(m)$ can be computed as:
$$
l(m):=\sum_{Q\in B(X)} l_Q(m):=\sum_{Q\in B(X)}
\sum_{j=1}^{m-1}\frac{\overline{jb_Q}(r_Q-\overline{jb_Q})}{2r_Q}
$$
where the sum $\sum_Q$ runs through all single baskets $Q$ of $B(X)$
with type $\frac{1}{r_Q}(1,-1,b_Q)$ and $\overline{jb_Q}$ means the
smallest residue of $jb_Q$ mod $r_Q$.
\end{thm}

\begin{rem} (1) It is clear that $\chi_{-m}=- \chi_{m+1}$ by duality.

(2) If  $K_X$ is nef and big, then $\chi_m=P_m$ for all $m \ge 2$.

(3) If $-K_X$ is nef and big, then $\chi_{-m}=P_{-m}$ for all $m
\ge 1$.
\end{rem}



One can notice that all the $\chi_m$ are determined by the triple
$$
(B(X), \chi(\OO), \chi_2).
$$
Therefore a triple $(B,\tilde{\chi},\tilde{\chi_{2}})$ with $B$ a basket
and integers $\tilde{\chi}$, $\tilde{\chi_{2}}$ is called a {\it
formal basket}. Given a formal basket, one can define all $\chi_m$
and $K^3$ formally by the Riemann--Roch formula of Reid. We write
$$
(B,\tilde{\chi},\tilde{\chi_{2}})\succ (B',\tilde{\chi}',\tilde{\chi_{2}}')
$$
if $ B\succ B'$ and $\tilde{\chi}=\tilde{\chi}',\tilde{\chi_{2}}=\tilde{\chi_{2}}'$.



One can try to recover formal baskets with given Euler
characteristics $\chi_m$. This was done in \cite{explicit}. For our
purpose in this note, we are only concerned with the initial formal
basket.
Assume that $B^{(0)}(B)=\{n^0_{1, r}\times (1,r)\}_{r\ge 2}.$
Keep the notation as in \cite{explicit}. Then one gets $B^{(0)}$ as
follows:
$$
\renewcommand{\arraystretch}{1.3}
\left\{
\begin{array}{rcl}
n^0_{1,2}&=&5\tilde{\chi}+6\tilde{\chi}_{2}-4\tilde{\chi}_{3}+\tilde{\chi}_{4}\\
n^0_{1,3}&=&4\tilde{\chi}+2\tilde{\chi}_{2}+2\tilde{\chi}_{3}-3\tilde{\chi}_{4}+\tilde{\chi}_{5}\\
n^0_{1,4}&=&\tilde{\chi}-3\tilde{\chi}_{2}+\tilde{\chi}_{3}+2\tilde{\chi}_{4}-\tilde{\chi}_{5}-\sigma_5\\
n^0_{1,r}&=&n^0_{1,r} \quad\hbox{for $r\ge 5$},
\end{array} \right.
$$
where $\sigma_5:=\sum_{r\ge 5} n^0_{1,r}$. A direct computation (cf.\
\cite[Lemma 4.13]{explicit}) gives:
\begin{eqnarray*}\epsilon_5&:=&\Delta^5(B^{(0)})-\Delta^{5}(B)\\
&=&2\tilde{\chi}-\tilde{\chi_3}+2\tilde{\chi_5}-\tilde{\chi_6}-\sigma_5.
\end{eqnarray*}
Therefore we get $B^{(5)}$ as follows:
$$
\renewcommand{\arraystretch}{1.3}
\left\{
\begin{array}{rcl}
n^5_{1,2}&=&3\tilde{\chi}+6\tilde{\chi_2}-3\tilde{\chi_3}+\tilde{\chi_4}-2\tilde{\chi_5}+\tilde{\chi_6}+\sigma_5\\
n^5_{2,5}&=&2\tilde{\chi}-\tilde{\chi_3}+2\tilde{\chi_5}-\tilde{\chi_6}-\sigma_5\\
n^5_{1,3}&=&2\tilde{\chi}+2\tilde{\chi_2}+3\tilde{\chi_3}-3\tilde{\chi_4}-\tilde{\chi_5}+\tilde{\chi_6}+\sigma_5\\
n^5_{1,4}&=&\tilde{\chi}-3\tilde{\chi_2}+\tilde{\chi_3}+2\tilde{\chi_4}-\tilde{\chi_5}-\sigma_5\\
n^5_{1,r}&=&n^0_{1,r} \quad\hbox{for $r\ge 5$}.
\end{array} \right.
$$


All the above formulas are parallel to those in \cite[Section
5]{explicit}. Noticing that, given $\tilde{\chi}$, $\tilde{\chi}_{2}$,
$\tilde{\chi}_{3}$, $\tilde{\chi}_{4}$ and $\tilde{\chi}_{5}$, we cannot
really determine $B^{(0)}$ completely. But we have $$\sigma:=\sum_{r}
n^0_{1,r}=10 \tilde{\chi} +5\tilde{\chi}_{2}- \tilde{\chi}_{3},$$ and
explicit $n^0_{1,2}$ and $n^0_{1,3}$. We consider the basket
$B'^{(0)}:=\{n'_{1,r} \times (1,r)\}$ with
$$
\renewcommand{\arraystretch}{1.3}
\left\{ \begin{array}{l}
n'_{1,r}=n^0_{1,r}, \quad\hbox{for $r=2,3,4$}, \\
n'_{1,5}=\sigma_5.
\end{array} \right.
$$









\begin{setup}
{\bf Poincar\'e series.} Consider $X=X_{d_1,\dots,d_c}$ in
$\PPP(a_0,\dots,a_n)$. The Poincar\'e series (see \cite[3.4]{WPS})
corresponding to the coordinate ring $R$ of $X$ is:
$$
\mathcal{P}(t)=\sum_{m=0}^\infty
h^0(X,\OO_X(m))t^m=\frac{\Pi_{j=1}^{c}(1-t^{d_j})}{\Pi_{i=0}^{n}(1-t^{a_i})}. \eqno(5.1)
$$
We may factorize this as $\mathcal{P}(t) = \frac{\prod h_i(t)}{ \prod
g_j(t) } $ with $h_i$, $g_j$ monic irreducible (modulo $\pm$ 1)  and
$h_i \ne g_j$. Notice that all the $h_i$, $g_j$ are cyclotomic
polynomials since they divide $1-t^n$ for some $n$.

We mainly consider the following two cases.
\medskip

\noindent{\bf Case 1.}  $\alpha=1$ and $\OO_X(K_X) =
\OO_X(1)$.

One has
$$
\mathcal{P}(t)=1+p_g(X)t+\sum_{m=2}^\infty P_m(X)t^m.
$$
Observing the correspondence to Riemann--Roch formula, we have
$$
\mathcal{P}(t)=1+p_g(X)  t+\frac{f_0(t)}{(1-t)^4}+\sum_Q \frac{h_Q(t)}{1-t^{r_Q}}t, \eqno(5.2)
$$
where
$$
f_0(t)={(1-t)^4}\sum_{t=2}^\infty \Bigl(\frac{m(m-1)(2m-1)}{12}K^3  +(1-2m) \chi\Bigr) t^m
$$
and
$$
h_Q(t)= \sum_{j=1}^{r_Q-1} \overline{j b_Q}(r_Q-
\overline{j b_Q})t^j \quad\text{for each}\quad Q=(b_Q,r_Q)
$$
are polynomials.
\medskip

\noindent{\bf Case 2.}  $\alpha=-1$ and $\OO_X(K_X) =
\OO_X(-1)$.

Consider
$$
\mathcal{P}(t)=1+\sum_{m=1}^\infty P_{-m}(X)t^m=1-\sum_{m=1}^\infty \chi((m+1)K_X)t^m.
$$
Similarly, we have
$$
\mathcal{P}(t)=1-\frac{f_0(t)}{t(1-t)^4}-\sum_Q \frac{h_Q(t)}{1-t^{r_Q}}, \eqno(5.3)
$$
where $f_0(t)$ and $h_Q(t)$ have the same form as in Case 1.
\end{setup}


By comparing the expression (5.1) with (5.2) (resp.\ with (5.3)), we have:

\begin{lem}\label{comparison} Let $X$ be a threefold weighted complete
intersection with $\alpha=\pm 1$. Suppose that  $a:=\max\{a_i\} \ge 2$
and $ r=\max\{r_Q\}$. Then either $a \le r$ or $a\mid d_j$ for some $j$.
\end{lem}

\begin{proof}
Suppose that $a>r$, then one sees that $(1-t^a)$ does not appear in the
denominator of $\mathcal{P}(t)$ by considering $(5.2)$ or $(5.3)$. This
means in particular that the cyclotomic polynomial $\varphi_a$ does not
appear in the denominator of $\mathcal{P}(t)$. By consider $(5.1)$, this
implies that $a\mid d_j $ for some $j$.
\end{proof}

\begin{setup}
{\bf The Reid table method.} Given a formal basket ${\bf B}$, one
can compute all the plurigenera (resp.\ anti-plurigenera). In
\cite[\S 18]{Fletcher}, Iano-Fletcher introduced the so-called Reid
table method which determines whether there exists a weighted
complete intersection with $\omega=\OO_X(\pm 1)$ with given
plurigenera or anti-plurigenera.

We recall the following:

\begin{lem}(\cite[18.3]{Fletcher}) Given a sequence $p_0=1, p_1, p_2,\dots$ such that
$$
\sum_{i=0}^\infty p_it^m=\frac{\Pi_{j=1}^{c}(1-t^{d_j})}{\Pi_{i=0}^{n}(1-t^{a_i})}
$$
for some pairs of positive integers $\{a_i,d_j\}$. Then those pairs
$\{a_i,d_j\}$ are unique up to $a_i\neq d_j$ and can be determined.
\end{lem}

Therefore, given a formal basket ${\bf B}$, one can determine whether
there exists a weighted complete intersection with $\omega_X=\OO_X(\pm
1)$ with given formal basket ${\bf B}$. To make the table method into an
algorithm, we need the following:

\begin{claim} Given a formal basket ${\bf B}$, there exists a
constant $M=M({\bf B})>0$ such that, if ${\bf B}$ is the formal basket
of certain threefold weighted complete intersection $X_{d_1,\dots,d_c}
\subset \PPP(a_0,\dots,a_n)$ with $a_i \neq d_j$ for each $i,j$,
one has $a_i$, $d_j \le M$.
\end{claim}

Therefore, one can make the table method for baskets into an
algorithm since one only needs to compute $P_m$,$a_i$,$d_j$ up to
$M$ for a given formal basket ${\bf B}$. The claim is true, at least
when $\alpha=\pm 1$, due to the following:

\begin{lem}\label{tmbound} Assume that $X=X_{d_1,\dots,d_c}\subset
\PPP(a_0,\dots,a_n)$ is a normalized weighted complete
intersection threefold having the formal basket
 ${\bf B}=(\chi,\chi_2,B=\{(b_i,r_i)\})$. For  $\alpha=\pm 1$, let $r:=\max\{r_i \} $
 and $s:=\max\{(\frac{4+c+\alpha}{c})^c\}_{c=1,\dots,4}$. Then  we
 have $d_c\leq 2N$ where
$$
N:=\max\{r, \lceil 1680s \rceil+\alpha, \}
$$
and hence all the $d_j$ and $a_i$ are bounded above by $M:=2N$.
\end{lem}

\begin{proof} By Theorem~\ref{cod}, we may assume that $c \le 5$. In fact, here we only assume that $c \le 4$ since
the case $c=5$ only occurs as $X_{2,2,2,2,2} \subset \PPP^9$
by virtue of Corollary~\ref{cod45}(2).

Suppose that $a_n > N$. In particular $a_n >r$, then $a_n\mid d_j$ for
some $j$ by Lemma~\ref{comparison}. It follows that $\delta_c \ge
d_j-a_n\ge a_n$. Moreover, if $\alpha >0$, then $X$ is of general
type with $\chi \le 1$. By \cite{explicit} or \cite{Zhu}, we have
 $K^3(X) \ge \frac{1}{420}$. By (2.6), we get
$$\frac{s}{a_0a_1a_2a_3}  \ge K^3(X) \ge \frac{1}{420}. $$
It follows that $a_0a_1a_2a_3 \leq 420s$. In particular, $a_i \le
420s$ for $i=0,\dots,3$.

If $\alpha <0$, then by \cite{AIM}, we have $-K^3 \ge \frac{1}{330}$
and, similarly, the inequality $a_0a_1a_2a_3 \le 330s$.

In total, $$a_n \le \delta_c \le
\delta=\delta_1+\cdots+\delta_4=a_0+\cdots+a_3+\alpha \le
1680s+\alpha \le N,$$ a contradiction. Hence $a_n\leq N$.

Finally, one has
$$d_c=\delta_c+a_n\leq N+a_n \leq 2N.$$
This completes the proof.
\end{proof}
\end{setup}

\section{\bf Part II--3: Weighted terminal $\bQ$-Fano threefolds }

Here we would like to classify all quasi\-smooth terminal complete
intersection $\bQ$-Fano threefolds
$X=X_{d_1,\dots,d_c}\subset\PPP=\PPP(a_0,\dots,a_{c+3})$
with $\alpha=-1$.

We use the algorithm proposed in Section~1. It suffices to check
that we have explicit boundedness for each step.

\medskip\paragraph{\bf Step~1.} To recall our definition in Introduction,
$\mu_i:=\#\{a_j\bigm|a_j=i\}$ and $\nu_i:=\#\{d_j\bigm|d_j=i\}$ for each
$i>0$. Since we have codimension $c\leq 3$ by Theorem~\ref{cod}, we
have
$$
\sum_{i=1}^5 \mu_i \le 7 \quad\hbox{and}\quad \sum_{i=2}^5 \nu_i \le 3.
$$
Thus the set of tuples $(\mu_1,\dots,\mu_5;\nu_2,\dots,\nu_5)$
is clearly finite.

\medskip\paragraph{\bf Step~2.} Given a tuple
$(\mu_1,\dots,\mu_5;\nu_2,\dots,\nu_5)$, one can compute
$P_{-m}=\chi_{-m}=-\chi_{m+1}$ for $m=1,\dots,5$. Thus one can
partially detect $B^0=\{n^0_{1,r} \times (1,r)\}$ as the
following:
$$
\renewcommand{\arraystretch}{1.3}
\begin{array}{rcl}
n^0_{1,2}&=&5\tilde{\chi}+6\tilde{\chi}_{2}-4\tilde{\chi}_{3}+\tilde{\chi}_{4},\\
n^0_{1,3}&=&4\tilde{\chi}+2\tilde{\chi}_{2}+2\tilde{\chi}_{3}-3\tilde{\chi}_{4}+\tilde{\chi}_{5} \\
\sigma=\sum n^0_{1,r}&=&10 \tilde{\chi}+5 \tilde{\chi}_2-\tilde{\chi}_3
\end{array}.$$

On the other hand, by \cite{KMMT}, we have $-K_X\cdot c_2(X)\ge 0$.
Then \cite[10.3]{YPG} and also \cite{ABR} give the inequality
$$
\sum_{i=1}^t \Bigl(r_i-\frac{1}{r_i}\Bigr) \le 24.
$$
The first consequence is that $r_i\le 24$ for all $i$. Notice
that $r-\frac{1}{r} \ge \frac{3}{2}$ for $r \ge 2$, we get $t \le
16$. Therefore,
$$
\sum r_i \le 24+ \sum \frac{1}{r_i} \le 24+\sum \frac{1}{2} \le 32.
$$
This already says that $\sigma$ is necessarily bounded above.

Therefore, we have finitely many possible $B^0$ and thus finitely
many possible formal baskets {\bf B}.

\medskip\paragraph{\bf Step~3.} For each formal basket, one can easily
compute the Poincar\'e series up to degree $M$ (cf.\
Lemma~\ref{tmbound}). The Reid table method can determine whether it
matches with a weighted complete intersection with  $a_i,d_j \le M$.
This completes the algorithm.

\medskip

Our conclusion on weighted $\bQ$-Fano threefolds with $\alpha=-1$ is
the following:

\begin{thm}\label{Fano} Iano-Fletcher's lists
\cite[16.6, 16.7, 18.16]{Fletcher} for weighted terminal $\bQ$-Fano
threefolds are complete. \end{thm}

\section{\bf Part II--4: Weighted terminal threefolds of general type}

In this section,  we explain how to classify all the quasi\-smooth
terminal threefold $X=X_{d_1,\dots,d_c}\subset \PPP(a_0,\dots,a_{c+3})$
with $\alpha=1$. The idea is similar to the previous section, with a
little further analysis in certain delicate cases. Note first that
$\OO_X(K_X)=\OO_X(1)$ is ample and $h^1(\OO_X)=h^2(\OO_X)=0$. Thus we
obtain the following easy but useful inequalities.


\begin{lem} \label{pluri}
Suppose $K_X$ is nef and big and $h^1(\OO_X)=h^2(\OO_X)=0$. Then
$$
P_{m+2} \ge P_m+P_2+p_g. \eqno(7.1)
$$
\end{lem}

\begin{proof}
Since $l(m+2) \ge l(m)+l(2)$, the Riemann--Roch formula and $K^3>0$ give
directly that
\begin{align*}
P_{m+2}-P_m-P_2&= (m^2+m)K^3-\chi+l(m+2)-l(m)-l(2)\\
 &>-\chi=p_g-1.
\end{align*}
\end{proof}

\begin{lem}\label{bf}
Suppose that  $K_X$ is nef and big and $h^1=h^2=0$. Then the
following holds:
$$
\frac{1}{12}\Bigl(1-p_g-P_2-P_3+P_5\Bigr)-\frac{1}{20}\sigma_5 >0.
\eqno(7.2)
$$
\end{lem}

%

\begin{proof}
Since $K^3(B_X)>0$ and by \cite[Lemma.6(3), Definition
5.3]{explicit}, one has $K^3(B'^0) \ge K^3(B^0) \ge K^3(B_X)$. A
direct computation on $K^3(B'^0)$ gives the inequality.
\end{proof}

We are now ready to justify our algorithm.

\medskip\paragraph{\bf Step~1.} We list all the possible tuples
$\{(\mu_1,\dots,\mu_6;\nu_2,\dots,\nu_6)\}$ for which $\sum \mu_i \le 9$
and $\sum \nu_i \le 5$ by virtue of Theorem~\ref{cod}. This set is, of
course, finite.

Notice that $\mu_i\nu_i=0$ for all $i\le 6$ by the assumption that
our weighted complete intersection is not a linear cone, i.e.,
$d_j\ne a_i$.

Moreover, if $\sum_{i=2}^6 \nu_i >0$, then in particular, $d_1 \le
6$. It follows that $a_i \le 5 $ for all $i \le 4$. Hence
$\sum_{i=1}^{5} \mu_i \ge 5$. More precisely, whenever
$\sum_{i=2}^6 \nu_i >0$, we have:
$$
\sum_{i=2}^s \nu_i  \le \sum_{i=1}^s \mu_i +4
\quad\hbox{for all}\quad 2\le s \le 6.
$$
We use these inequalities also to eliminate extra cases.

\medskip\paragraph{\bf Step~2.} This step is a bit more complicated.

We first estimate $\sigma_5$. Recall that
\begin{align*}
n^0_{1,4}&=\tilde{\chi}-3\tilde{\chi}_{2}+\tilde{\chi}_{3}+2\tilde{\chi}_{4}-\tilde{\chi}_{5}-\sigma_5  \ge 0,\\
\epsilon_5&=2\tilde{\chi}-\tilde{\chi}_{3}
+2\tilde{\chi}_{5}-\tilde{\chi}_{6}-\sigma_5 \ge 0.
\end{align*}
We set
$$
M(\sigma_5):=
\min\bigl\{\tilde{\chi}-3\tilde{\chi}_{2}+\tilde{\chi}_{3}+2\tilde{\chi}_{4}-\tilde{\chi}_{5},\,
2\tilde{\chi}-\tilde{\chi}_{3}+2\tilde{\chi}_{5}-\tilde{\chi}_{6}\bigr\},
$$
which is an upper bound for $\sigma_5$. Also notice that
$n^0_{1,2}$, $n^0_{1,3} \ge \epsilon_5$ (i.e., $n^5_{1,2}$,
$n^5_{1,3} \ge 0$) gives
$$
\renewcommand{\arraystretch}{1.3}
\begin{array}{rcl}
\sigma_5 &\ge& -3\tilde{\chi}-6\tilde{\chi}_{2}+3\tilde{\chi}_{3}-\tilde{\chi}_{4}+2\tilde{\chi}_{5}-\tilde{\chi}_{6}, \\
\sigma_5 &\ge& -2\tilde{\chi}-2\tilde{\chi}_{2}-3\tilde{\chi}_{3}+3\tilde{\chi}_{4}+\tilde{\chi}_{5}-\tilde{\chi}_{6}.
\end{array}
$$
We set $m(\sigma_5)$ to be the maximum
$$
\max\bigl\{0,-3\tilde{\chi}-6\tilde{\chi}_{2}+3\tilde{\chi}_{3}-\tilde{\chi}_{4}+2\tilde{\chi}_{5}-\tilde{\chi}_{6}
,-2\tilde{\chi}-2\tilde{\chi}_{2}-3\tilde{\chi}_{3}+3\tilde{\chi}_{4}+\tilde{\chi}_{5}-\tilde{\chi}_{6}
\bigr\},
$$
which is an lower bound of $\sigma_5$.

For a given tuple $(\mu_1,\dots,\mu_6;\nu_1,\dots,\nu_6)$.  We
can first compute $p_g$, $P_2,\dots,P_6$, which gives
$\chi, \chi_2,\dots,\chi_6$ directly. Hence one can determine:
$$
B'':=\{n''_{1,2} \times (1,2),n''_{1,3} \times (1,3), n''_{1,4} \times
(1,4)\},
$$
with
$$
\renewcommand{\arraystretch}{1.3}
\left\{
\begin{array}{lllll}
n''_{1,2}&=&5\tilde{\chi}+6\tilde{\chi}_{2}-4\tilde{\chi}_{3}+\tilde{\chi}_{4}&=&n^0_{1,2},\\
n''_{1,3}&=&4\tilde{\chi}+2\tilde{\chi}_{2}+2\tilde{\chi}_{3}-3\tilde{\chi}_{4}+\tilde{\chi}_{5}&=&n^0_{1,3},\\
n''_{1,4}&=&\tilde{\chi}-3\tilde{\chi}_{2}+\tilde{\chi}_{3}+2\tilde{\chi}_{4}-\tilde{\chi}_{5}&=&n^0_{1,4}+\sigma_5.
\end{array} \right. $$

Recall that
$$
\renewcommand{\arraystretch}{1.3}
\begin{array}{rcl} B^0&=&\{n^0_{1,2} \times
(1,2),n^0_{1,3} \times (1,3), n^0_{1,4} \times (1,4),
(1,r_1),\dots,(1,r_{\sigma_5})\}, \\
B'^0&=&\{n^0_{1,2} \times (1,2), n^0_{1,3} \times (1,3), n^0_{1,4}
\times (1,4), \sigma_5 \times (1,5)\}. \end{array}
$$
Notice that
$B'' \succ B'^0 \succ B^0$. We shall show that there are only
finitely many possible $B^0$ from the given tuple.

To eliminate some impossible cases, we use (7.1) together with the
following inequalities:
$$
\renewcommand{\arraystretch}{1.3}
\left\{\begin{array}{ll} P_m \ge 0, & P_{2m} \ge 2P_m-1,\\
n^0_{1,2} \ge 0, &
n^0_{1,3} \ge 0,\\
n^0_{1,4}+\sigma_5 \ge 0,\\
M(\sigma_5) \ge m(\sigma_5),\\
\frac{1}{12}(1-p_g-P_2-P_3+P_5)-\frac{1}{20} m(\sigma_5) >0.
\end{array} \right.
$$

We now proceed to distinguish several cases. In fact, our computation
shows that one of the following situations occurs.

\medskip\paragraph{\bf Case 2.1.} $M(\sigma_5)=0$.

Clearly, one has $\sigma_5=0$. It follows that $B^0=B''$ can be
determined. Hence there are finitely many possible formal baskets
dominated by $B^0$ (by considering all possible prime packings).

\medskip\paragraph{\bf Case 2.2.}
$t:=\frac{1}{12}\bigl(1-p_g-P_2-P_3+P_5\bigr)-\frac{1}{20}m(\sigma_5)
<\frac{1}{4}$.

In this situation, by the computation in Lemma~\ref{bf}, we have:
$$
0 < K^3(B^0) = K^3(B'^0) + \sum_{i=1}^{\sigma_5}
\Bigl(\frac{1}{r_i}-\frac{1}{4}\Bigr)\le t+ \sum_{i=1}^{\sigma_5}
\Bigl(\frac{1}{r_i}-\frac{1}{4}\Bigr).
$$
It follows in particular that $r_i
\le \frac{1}{\frac{1}{4}-t}$ for each $r_i$. Therefore, we have
finiteness of initial baskets $B^0$ and hence formal baskets.

\medskip\paragraph{\bf Case 2.3.}
$\sum_{i=1}^6 \mu_i \ge 5$ or $\nu_i >0$ for some $i$.

Since $\sum_{i=1}^6 \mu_i \ge 5$, we have $\delta \le 25$. Then $d_1 \le
31$. We conclude that $a_n \le 31$ since otherwise, by
Proposition~\ref{qs}, $a_n\mid d_c$ and then $\delta_c \ge a_n \ge 31$,
a contradiction. If $\nu_i >0$ for some $i$ then, by Step~1,
$\sum_{j=1}^i \mu_j \ge 5$. We still have $a_n \le 31$.

Recall that all singularities have index $h$ which is a greatest
common divisor of some of the $\{a_i\}$. It follows that  $r_i \le
a_n \le 31$ for all $i$. Thus we are able to classify initial
baskets.



\medskip\paragraph{\bf Step~3.} Once we classified formal baskets and
computed their Euler characteristic, we could run the table method.
This already justifies our algorithm.
\bigskip

A byproduct of our computation is the following:

\begin{cor}\label{cod45} (1) A canonically polarized threefold that is a
quasi\-smooth codimension~$4$ complete intersection must be
$X_{2,2,2,3}\subseteq \PPP^7$.

(2) A canonically polarized threefold that is a quasi\-smooth
codimension~$5$ weighted complete intersection must be
$X_{2,2,2,2,2} \subset \PPP^8$.
\end{cor}

The reader familiar with weighted projective space should find it an
amusing exercise to prove the above mentioned results by utilizing
Proposition~\ref{qs}, properties of well-formedness and singularity
computing. In fact, Corollary~\ref{cod45} is covered by our general
analysis in the context and so we omit the proof.

To summarize our main result, we have established the following:

\begin{thm}\label{gt123} Iano-Fletcher's lists \cite[15.1, 15.4,
18.16]{Fletcher} are complete.
\end{thm}
\bigskip

\noindent{\bf Acknowledgment.}  We would like to thank Gavin Brown,
Christopher Hacon, Chin-Lung Wang for helpful discussions. We are
very grateful to Miles Reid who had spent a lot of time reading and
even rewriting part of the paper. It is his suggestions and skillful
comments that helps us to understand the relevant topic. We would
like to express our immense gratitude to Miles Reid.

\end{document}